\documentclass{amsart}
\newtheorem{Theo}{Theorem}
\newtheorem{Lem}{Lemma}

\newcommand{\R}{\mathbb{R}}
\begin{document}
\title[Continuous prime systems]{Continuous prime systems satisfying $N(x)=c(x-1)+1$}
\author[J.-C. Schlage-Puchta]{Jan-Christoph Schlage-Puchta}
\begin{abstract}
Hilberdink showed that there exists a constant $c_0>2$, such that there exists a continuous prim system satisfying $N(x)=c(x-1)+1$ if and only if $c\leq c_0$. Here we determine $c_0$ numerically to be $1.25479\cdot 10^{19}\pm2\cdot 10^{14}$. To do so we compute a representation for a twisted exponential function as a sum over the roots of the Riemann zeta function. We then give explicit bounds for the error obtained when restricting the occurring sum to a finite number of zeros.
\end{abstract}

\maketitle

MSC-Index 11N80, 11Y60, 30A10, 33E20, 65E05

Jan-Christoph Schlage-Puchta\\
Mathematisches Institut, Universit\"at Rostock\\
Ulmenstra\ss e 69, Haus 3\\
18057 Rostock\\
Germany\\
\verb+jan-christoph.schlage-puchta@uni-rostock.de+
\pagebreak

\section{Introduction and results}
Let $S$ be the space of right-continuous functions $f:\R\rightarrow\R$ of bounded local variation, for which $f(x)=0$ for $x<1$. Let $S^+$ be the subset consisting of non-decreasing functions. For functions $f, g\in S$ define the Mellin-Stiltjes convolution $f\ast g$ by means of the equation
\[
(f\ast g)(x) = \int\limits_{1^-}^xf(x/t)dg(t),
\]
and the convolution exponential $\exp_\ast g$ as
\[
\exp_\ast g = \sum_{n=0}^\infty\frac{g^{\ast n}}{n!},
\]
where $g^{\ast n}$ denotes $n$-fold iterated convolution.

For $\pi\in S^+$ define $\Pi(x)=\sum_{k\geq 1}\frac{1}{k}\pi(x^{1/k})$ and $N=\exp_\ast \Pi$. If the sum defining $\Pi$ converges for all $x$, then we call the pair $(\Pi, N)$ a continuous prime system with prime counting function $\pi$. Note that if $\pi(x)$ denotes the number of ordinary primes below $x$, we obtain $N(x)=\lfloor x\rfloor$, and $\Pi(x)$ is the weighted number of prime powers below $x$ introduced by Riemann. If more generally $\pi(x)$ is a step function with integral jumps, then $N(x)$ is the counting function of an arithmetic semigroup in the sense of Knopfmacher\cite{Knopfmacher}.

Starting with the work of Beurling there has been ongoing interest in continuous prime systems. Hilberdink\cite{Titus} showed that if there is some $c$, such that $N(x)-cx$ is periodic and continuously differentiable, then $N(x)=c(x-1)+1$. This led him to ask, for which $c$ such a number system exists. Define the holomorphic function $f$ as
\[
f(z)=\sum_{n=1}^\infty\frac{\mu(n)}{n}(e^{z/n}-1) = \sum_{k=1}^\infty\frac{z^k}{k!\zeta(k+1)}.
\]
He then proved the following.
\begin{Theo}
\label{thm:main}
There exists a continuous prime system satisfying $N(x)=c(x-1)+1$ if and only if $f(x)\geq f((1-c)x)$ for all $x\geq 0$. Moreover, there exists some $c_0>2$ such that there exists such a prime system if and only if $c\leq c_0$. 
\end{Theo}
Here we determine $c_0$ numerically. Clearly the existence of $c_0$ is equivalent to the statement that $f(x)$ is positive for some $x<0$. Hilberdink proved the existence of such an $x$ using Landau's ineffective criterion on the continuation of Dirichlet series with non-negative coefficients, therefore his proof does not yield any bound on $c_0$.

We prove the following.

\begin{Theo}
\label{thm:Numerik}
The constant $c_0$ from Theorem~$\ref{thm:main}$ satisfies
\[
\left|c_0-1.25479\cdot 10^{19}\right|\leq 2\cdot 10^{14}.
\]
\end{Theo}

\section{Asymptotic estimates for $f$}

In the sequel $\theta$ denotes a complex number of modulus $\leq 1$, which may be different in all equations and may depend on all occurring parameters. As in the case of Landau symbols, equations containing $\theta$ may only be read from left to right, e.g. we have $\theta=2\theta$, but not $2\theta=\theta$.

The following is a version of Stirling's formula with an explicit error term,
derived by Boyd\cite{Boyd}.
\begin{Lem}
\label{Lem:Gamma bound}
For $|\arg z|\leq\frac{\pi}{2}$ we have
\[
\Gamma(z) = \sqrt{2\pi z}\left(\frac{z}{e}\right)^z\left(1+\theta\frac{1+\sqrt{2}}{2\pi^2|z|}\right)
\]
\end{Lem}
We can now come to the main result of this section. We denote the non-trivial roots of $\zeta$ by $\rho$, and the imaginary part of $\rho$ by $\gamma$.
\begin{Lem}
\label{Lem:explicit}
Let $T\geq 100$ be a real number such that all roots of $\zeta$ in the rectangle
$0\leq\sigma\leq 1$, $|t|\leq T$ are simple with real part $\frac{1}{2}$, and
that $\zeta$ has no root with imaginary part $T$. Put
$\delta=\min_{-2\leq\sigma\leq 2}|\zeta(\sigma+iT)|$. Then we have for real
$x>e^2$ the estimate
\begin{multline}
\label{eq:explicit}
f(-x) = \frac{1}{x^2\zeta'(-1)}+\frac{1}{\sqrt{x}}\sum_{|\gamma|<T}
\frac{\Gamma(1-\rho)}{\zeta'(\rho)} x^{i\gamma}\\ +
\theta\frac{15.18}{x^{5/2}} + \theta(0.85\log x+0.88\delta^{-1})T^2
e^{-\pi T/2} 
\end{multline}
\end{Lem}
\begin{proof}
From the Mellin transform
\[
\frac{1}{2\pi i}\int\limits_{-\frac{1}{2}-i\infty}^{-\frac{1}{2}+i\infty} \Gamma(s)x^{-s}ds = e^{-x}-1
\]
we deduce
\[
f(-x)=\frac{1}{2\pi i}\int\limits_{\frac{3}{2}-i\infty}^{\frac{3}{2}+i\infty}\frac{\Gamma(1-s)}{\zeta(s)}x^{s-1}ds.
\]
We shift the path of integration to the path going from $1+\frac{1}{\log x}-i\infty$ to $1+\frac{1}{\log x}-iT$, then to $-\frac{3}{2}-iT$, to $-\frac{3}{2}+iT$, further to $1+\frac{1}{\log x}+iT$, and finally to $1+\frac{1}{\log x}+i\infty$.  Doing so we encounter one singularity at $s=-1$ with residuum $\frac{1}{x^2\zeta'(-1)}$, and one singularity with residuum $\frac{\Gamma(1-\rho)}{\zeta'(\rho)} x^{1/2+i\gamma}$ for each non-trivial root $\rho$ in the rectangle $0\leq\sigma\leq 1$, $|t|< T$. Note that the pole of $\zeta$ at 1 and the pole of $\Gamma$ at 0 cancel each other. The integral over the new path will be bounded from above. We have
\begin{eqnarray*}
\left|\int\limits_{-\frac{3}{2}-iT}^{-\frac{3}{2}+i T}\frac{\Gamma(1-s)}{\zeta(s)}x^{s-1}ds\right| & \leq & \frac{1}{x^{5/2}}\int\limits_{-\frac{3}{2}-i\infty}^{-\frac{3}{2}+i\infty}\left|\frac{\Gamma(1-s)}{\zeta(s)}\right|ds,
\end{eqnarray*}
and since $\Gamma$ decreases rapidly along every line parallel to the
imaginary axis, the last integral can easily be evaluated numerically to be
$\leq 95.32$. 

On the line $\Re\;s=1+\frac{1}{\log x}$ we have
\[
\frac{1}{|\zeta(s)|}<\zeta(1+\frac{1}{\log x})<1+\int_1^\infty\frac{dt}{t^{1+1/\log x}} = 1+\log x,
\]
thus
\begin{eqnarray*}
\left|\int\limits_{1+\frac{1}{\log x}+i}^{1+\frac{1}{\log x}+i\infty}\frac{\Gamma(1-s)}{\zeta(s)}x^{s-1}ds\right| & \leq &e(1+\log x)\int\limits_{1+\frac{1}{\log x}+iT}^{1+\frac{1}{\log x}+i\infty}\left|\Gamma(1-s)\right|ds,
\end{eqnarray*}
and from Lemma~\ref{Lem:Gamma bound} we obtain that for $x\geq e^2$ the right hand side is bounded above by
\[
e(1+\log x)\int_T^\infty (t+1)e^{-\pi t/2}\;dt = e(1+\log x)(\frac{2}{\pi}T+\frac{4+2\pi}{\pi^2})e^{-\pi T/2}.
\]

Finally we have
\[
\left|\int\limits_{-\frac{3}{2}+iT}^{1+\frac{1}{\log x}+iT}\frac{\Gamma(1-s)}{\zeta(s)}x^{s-1}ds\right| \leq \delta^{-1}(T+1)e^{-\pi T/2}\int\limits_{-\frac{3}{2}}^{1+\frac{1}{\log x}} x^{\sigma-1}d\sigma\leq e(T+1)e^{-\pi T/2}\delta^{-1} .
\]
We conclude that the modulus of the integral over the new path is bounded above by
\begin{multline*}
\frac{95.32}{x^{5/2}} + (1+\log x)(3.462T+5.665\big)e^{-\pi T/2} + 2e(T+1)\delta^{-1}e^{\pi T/2}\\
\leq \frac{95.32}{x^{5/2}} + (5.491\delta^{-1}+5.279\log x)Te^{-\pi T/2},
\end{multline*}
where we used the bounds $1+\log x\leq\frac{3}{2}\log x$ and $T\geq 100$.
Taking the factor $\frac{1}{2\pi}$ into account our claim follows.\end{proof}

Note that even if we assume RH and the simplicity of all roots, we cannot get an explicit formula depending only on $x$ and $T$, since it might be that $\zeta'(\rho)$ could be very close to 0. However, as in the explicit formula for $\sum_{n\leq x}\mu(n)$, we do get an explicit formula valid for all suitable values of $T$. We refer the reader to \cite[section 14.27]{Titchmarsh} for details.

\begin{Lem}
\label{Lem:absolute bound}
We have $f(-x)<0$ for $0<x<2.5\cdot 10^6$, and $f(-x)<9.2\cdot 10^{-13}$ for all $x>0$.
\end{Lem}
\begin{proof}
We claim that in the range $7<x<2.5\cdot 10^6$ the first negative summand in (\ref{eq:explicit})
dominates the other terms. We put $T=100$. A straightforward computation yields $\delta\geq 1.19$,
together with $\zeta'(-1)=-0.165421\dots$ we obtain
\begin{eqnarray*}
f(-x) & \leq &
-\frac{6.045}{x^2}+\frac{15.18}{x^{5/2}}+\frac{1}{\sqrt{x}}\sum_{|\gamma|\leq 100}
\left|\frac{\Gamma(1-\rho)}{\zeta'(\rho)}\right| +(0.85\log x+ 0.74)\cdot
6.05\cdot 10^{-65}\\
 & \leq & -\frac{6.045}{x^2}+\frac{15.18}{x^{5/2}}+\frac{1.44\cdot
  10^{-9}}{\sqrt{x}} + (5.15\log x + 4.48)\cdot 10^{-65}.
\end{eqnarray*}
From this we conclude $f(-x)<0$ for $7<x<2.5\cdot 10^6$ as well as
$f(-x)<9.2\cdot 10^{-13}$ for $2.5\cdot 10^6<x<e^{10^{50}}$.

If $x$ is very big we use estimates for the summatory function of the M\"obius
function. We have
\[
\sum_{n=1}^\infty\frac{\mu(n)}{n}(e^{-x/n}-1) \leq \sum_{n=1}^\infty m(n)\left|e^{-x/n}-e^{-x/(n+1)}\right|,
\]
where $m(x)=\left|\sum_{n\leq x}\frac{\mu(n)}{n}\right|$.
Bordell\'es\cite{Bordelles} has shown that $m(x)\leq \frac{546}{\log^2 x}$ for
$x>1$, hence for $x>e^{24}$ we get
\begin{eqnarray*}
|f(-x)| & \leq & e^{-x} + \sum_{n=2}^\infty \frac{546}{\log^2
  n}e^{-x/n}\left|e^{-x/(n(n+1))}-1\right|\\
 & \leq & e^{-x} + \sum_{n=2}^\infty \frac{546}{\log^2 n} e^{-x/n}\min\left(\frac{2x}{n(n+1)}, 1\right)\\
 & \leq & xe^{-x^{1/3}} + \frac{1092}{\log^2x^{2/3}}\sum_{n\geq x^{2/3}}\frac{1}{n^2}\\
 & \leq & e^{24-e^8} + \frac{4.27}{x^{2/3}},
\end{eqnarray*}
which is sufficiently small for $x>10^{19}$.

In the range $\frac{1}{2}\leq x\leq 7$ we can compute $f$ with high precision using its Taylor series. We have
\[
|f''(-x)| = \left|\sum_{n=1}^\infty \frac{\mu(n) e^{-x/n}}{n^3}\right| \leq \sum_{n=1}^\infty \frac{e^{-x/n}}{n^3} \leq e^{-x} +\int_0^\infty \frac{e^{-x/t}}{t^3}\;dt = e^{-x}+\frac{1}{x^2}.
\]
Thus for a given $x_0$, we compute $f(x_0)$ and $f'(x_0)$, estimate $f''(x_0)$, and obtain an interval for which $f$ is negative. Finally in the range $0<x\leq \frac{1}{2}$ we have 
\[
f'(-x) = \sum_{k=1}^\infty\frac{(-x)^{k-1}}{(k-1)!\zeta(k+1)} \geq \frac{1}{\zeta(2)}-\frac{x}{\zeta(3)}>0,
\]
together with $f(0)=0$ we conclude that $f(-x)<0$ in $0<x\leq \frac{1}{2}$ as well.

Hence the lemma is proven for all $x>0$.
\end{proof}

\section{Computation of $c_0$}

The problem of computing $c_0$ is equivalent to finding the infimum of all $c$, such that there exists some $y\geq 0$ with $f(-y)\geq f(\frac{y}{c-1})$. Since $f(x)$ is increasing for $x\geq 0$, the right hand side is decreasing with $c$, hence our problem is equivalent to minimizing $\frac{x}{y}$ subject to the relations $x,y>0$, $f(x)=f(-y)$.

By Lemma~\ref{Lem:absolute bound} we have $f(-y)<9.2\cdot 10^{-13}$. As $f(x)\geq\frac{x}{\zeta(2)}$ for $x\geq 0$, the equation $f(-y)=f(x)$ implies $x<2\cdot 10^{-12}$. Together with $f(-y)<0$ for $y<2.5\cdot 10^6$ we obtain $\frac{x}{y}<5\cdot 10^{18}$ for all $x, y>0$ satisfying $f(x)=f(-y)$. This crude lower bound is surprisingly close to the actual value for $c$.

For two positive real numbers $y_1, y_2$ we say that $y_1$ is better than $y_2$, if $f(-y_1)>0$, and either $f(-y_2)\leq 0$, or for the real numbers $x_1, x_2>0$ defined by the equation $f(-y_i)=f(x_i)$ we have $\frac{x_1}{y_1}<\frac{x_2}{y_2}$. Clearly if $y_1$ is better than $y_2$, then $y_2$ cannot solve our optimization problem. We first show that in this way the range of $y$ can be restricted to a bounded interval.

\begin{Lem}
\label{Lem:better}
Suppose that $x_1>0$ satisfies $f(-x_1)>0$. Then $x_1$ is better than all $x_2$ satisfying $x_2 > \frac{9.2\cdot 10^{-13}}{f(-x_1)} x_1$.
\end{Lem}
\begin{proof}
Suppose that $x_2>x_1$, and that $x_1$ is not better than $x_2$. Let $y_1, y_2$ be given by the equations $f(-x_i)=f(y_i)$. We then have $y_2>y_1$, and since $f$ is convex in $x\geq 0$, we conclude that $\frac{f(y_2)}{y_2}>\frac{f(y_1)}{y_1}$, thus $\frac{f(-x_2)}{x_2}>\frac{f(-x_1)}{x_1}$. Our claim now follows from Lemma~\ref{Lem:absolute bound}.
\end{proof}

We now apply Lemma~\ref{Lem:explicit} with $T=100$ and neglect all roots except $\frac{1}{2}+i\gamma_1$, where $\gamma_1=14.13\ldots$ to find
\begin{eqnarray*}
f(-x) & = & \frac{1}{x^2\zeta'(-1)}+\frac{2}{\sqrt{x}}\Re\frac{x^{i\gamma_1}\Gamma(\frac{1}{2}-i\gamma_1)}{\zeta'(\rho_1)}\\
&& + \theta\left(\frac{4\cdot 10^{-14}}{x^{1/2}} + 
\frac{15.18}{x^{5/2}} + (5.15\log x + 4.48)\cdot 10^{-65}\right)\\
 & = & \frac{1}{x^2\zeta'(-1)}+\frac{10^{-14}}{\sqrt{x}}\Re \left(x^{i\gamma_1}(-14102+143259i+5\theta)\right)\nonumber\\
&& + \theta\left(\frac{15.18}{x^{5/2}} + (5.15\log x + 4.48)\cdot 10^{-65}\right)\\
 & = & \frac{1}{x^2\zeta'(-1)}+\frac{10^{-14}}{\sqrt{x}}\Re \left(x^{i\gamma_1}(-14102+143259i+21\theta)\right),
 \end{eqnarray*}
 provided that $2.5\cdot 10^6\leq x\leq 10^{50}$. Putting $s=\log(-x)$, we
 obtain
\begin{equation}
 \label{eq:special explicit}
f(-e^s) =\frac{e^{-2s}}{\zeta'(-1)} + (143951+22\theta)\cdot 10^{-14} e^{-s/2}
\cos(\gamma_1 s+1.66892)
\end{equation}
In particular we obtain $f(-e^{15})>2.3\cdot 10^{-13}$, thus, using Lemma~\ref{Lem:better}, $e^{15}$ is better than all $x$ satisfying $x>4\cdot e^{15}$. In particular we only have to consider values of $x$, for which the approximation (\ref{eq:special explicit}) is valid. 

Considering the power series for $f$ we find that for $x\in[0, 10^{20}]$ with $f(-x)>0$ the unique value $y$ with $f(y)=f(-x)$ satisfies $y<\zeta(2)f(-x)$, as well as 
\[
f(y)<\frac{y}{\zeta(2)} + e^y-1-y < \frac{y}{\zeta(2)}+y^2,
\]
thus
\[
y>\zeta(2)f(-x)-(\zeta(2)f(-x))^2>\left(1-\frac{3\cdot 10^{-9}}{\sqrt{x}}\right)\zeta(2)f(-x),
\]
and therefore $y=\left(1+\frac{3\cdot 10^{-9}\theta}{\sqrt{x}}\right)\zeta(2)f(-x)$. We conclude that in the relevant range the function to be minimized is
\[
\left(1+\frac{4\cdot 10^{-9}\theta}{\sqrt{x}}\right)\frac{x}{f(-x)},
\]
subject to the condition $f(-x)>0$. Since this condition in particular implies that the first, negative, summand in (\ref{eq:special explicit}) is of smaller absolute value than the second, we obtain that we have to minimize the inverse of
\[
\frac{e^{-3s}}{\zeta'(-1)} + (143951+24\theta)\cdot 10^{-14} e^{-3s/2}
\cos(\gamma_1 s+1.66892)
\]
subject to the condition that this expression is positive, that is, we have to find the largest local maximum of this function.

The first positive local maximum of this function occurs at $s=14.99$ with a value of $7.01\cdot 10^{-20}$, the second at $15.44$ with a value of $7.97\cdot 10^{-20}$, the third at $15.88$ with a value $5.26\cdot 10^{-20}$. All further local maxima are much smaller. The precision is sufficient to guarantee that the maximum is attained in the interval $[15.43, 15.45]$ and has a value in the interval $[7.9\cdot 10^{-20}, 8\cdot 10^{-20}]$.

We can now refine our computation by using the latter bound to improve the error in (\ref{eq:explicit}). We put $T=100$ in Lemma~\ref{Lem:explicit} and get
\begin{eqnarray*}
f(-x) & = & \frac{1}{x^2\zeta'(-1)}+\frac{2}{\sqrt{x}}\sum_{j=1}^{29}
\frac{\Gamma(\frac{1}{2}-i\gamma_j)}{\zeta'(\frac{1}{2}+i\gamma_j)} \Re x^{i\gamma_j} +
1.69\cdot 10^{-16}\theta \\
 & = & \frac{1}{x^2\zeta'(-1)}+\frac{2}{\sqrt{x}}\left|\frac{\Gamma(\frac{1}{2}-i\gamma_j)}{\zeta'(\frac{1}{2}+i\gamma_j)}\right|\cos\left(\gamma_1\log x + \arg\frac{\Gamma(\frac{1}{2}-i\gamma_j)}{\zeta'(\frac{1}{2}+i\gamma_j)}\right) +
1.75\cdot 10^{-16}\theta \\
\end{eqnarray*}
for $e^{15.43}<x<e^{15.45}$. From this we find that the maximum of $\frac{f(-x)}{x}$ is attained in $\log x = 15.4382+\theta 0.0001$ and has a value $(796947+\theta)\cdot 10^{-25
}$, and the value of $c_0$ is $(1.25479+0.00002\theta)\cdot 10^{19}$. The proof of Theorem~\ref{thm:Numerik} is complete.

\end{document}